\documentclass[12pt]{article}

\title{Ehrhart polynomials and symplectic embeddings of ellipsoids}
\author{Daniel Cristofaro-Gardiner and Aaron Kleinman}
\date{}

\usepackage{amssymb}
\usepackage{latexsym}
\usepackage{amsmath}
\usepackage{amsthm}
\usepackage{amscd}
\usepackage{appendix}
\usepackage{graphicx}
\usepackage{enumerate}

\numberwithin{equation}{section}

\newtheorem{theorem}{Theorem}[section]
\newtheorem{proposition}[theorem]{Proposition}

\newtheorem{lemma}[theorem]{Lemma}
\newtheorem{lemma-definition}[theorem]{Lemma-Definition}

\theoremstyle{definition}

\newtheorem{remark}[theorem]{Remark}

\newcommand{\eqdef}{\;{:=}\;}

\newcommand{\C}{{\mathbb C}}

\newcommand{\R}{{\mathbb R}}
\newcommand{\N}{{\mathbb N}}
\newcommand{\Z}{{\mathbb Z}}
\newcommand{\op}{\operatorname}

\renewcommand{\u}{u}

\newcommand{\se} {{\stackrel{s}\hookrightarrow}}

\newcommand{\T}{{\mathcal T}}

\renewcommand{\N}{\mathcal{N}}

\begin{document}

\setcounter{tocdepth}{2}

\maketitle
 
\begin{abstract}
McDuff and Schlenk determined when a four-dimensional ellipsoid can be symplectically embedded into a ball, and found that part of the answer is given by a ``Fibonacci staircase.'' Similarly, Frenkel and M\"uller determined when a four-dimensional ellipsoid can be symplectically embedded into the ellipsoid $E(1,2)$ and found that part of the answer is given by a ``Pell staircase.'' ECH capacities give an obstruction to symplectically embedding one four-dimensional ellipsoid into another, and McDuff showed that this obstruction is sharp. We use this result to give new proofs of the staircases of McDuff-Schlenk and Frenkel-M\"uller, and we prove that another infinite staircase arises for embeddings into the ellipsoid $E(1,\frac{3}{2})$. Our proofs relate these staircases to a combinatorial phenomenon of independent interest called ``period collapse" of the Ehrhart quasipolynomial.  

In the appendix, we use McDuff's theorem to show that for $a \ge 6$ the only obstruction to embedding an ellipsoid $E(1,a)$ into a scaling of $E(1,\frac{3}{2})$ is the volume, and we also give new proofs of similar results for embeddings into scalings of $E(1,1)$ and $E(1,2)$.

\end{abstract}
 
\section{Introduction}
\subsection{Statement of results}
\label{sec:intro}
Recently, McDuff has proven a powerful theorem concerning when one four-dimensional {\em symplectic ellipsoid}\footnote{Here, the symplectic form is given by restricting the standard form $\omega=\Sigma_i^2 dx_idy_i$ on $\R^4=\C^2.$} 
\[ E(a,b)= \lbrace (z_1,z_2) \in \C^2 | \frac{\pi |z_1|^2}{a}+\frac{\pi |z_2|^2}{b} \le 1 \rbrace \] embeds into another.  To state McDuff's theorem, let $c_k(E(a,b))$ be the $(k+1)^{st}$ smallest element in the matrix of numbers 
\[(am + bn)_{m,n \in \Z_{\ge 0}}. \]    
The number $c_k(E(a,b))$ is the $k^{th}$ {\em embedded contact homology} (ECH) capacity\footnote{The embedded contact homology capacities are a sequence of nonnegative (possibly infinite) real numbers $c_k(X)$ defined for any symplectic four-manifold.  The ECH capacities obstruct symplectic embeddings, see \cite{echlecture} for a summary.} of $E(a,b)$. McDuff showed that the ECH capacities give sharp obstructions to symplectic embeddings of ellipsoids:  
\begin{theorem} \cite[Thm. 1.1]{hofer}
\label{thm:mcduffhofer}
\[
\op{Int} E(a,b) \se E(c,d),
\]
if and only if
\[
c_k(E(a,b)) \le c_k(E(c,d))
\]
for all $k$.
\end{theorem}
Here, the symbol $\se$ means that the embedding is symplectic, while $\op{Int} E(a,b)$ denotes the interior of $E(a,b)$.

Theorem~\ref{thm:mcduffhofer} connects lattice point enumeration with symplectic geometry\footnote{Connections between lattice point counts and symplectic geometry were previously observed by McDuff-Schlenk in \cite{ms} and Hutchings in \cite{qech}.}.  In the present work, we explore some of these connections.  Specifically, in \cite{ms} McDuff and Schlenk found that embeddings of a four-dimensional ellipsoid $E(1,a)$ into a ball are partly determined by an infinite ``Fibonacci staircase," and in \cite{fm} Frenkel and M\"uller studied embeddings of $E(1,a)$ into a scaling of $E(1,2)$ and discovered another staircase involving the Pell numbers\footnote{In fact, McDuff and Schlenk determine for all $a$ exactly when $E(1,a)$ symplectically embeds into a ball and Frenkel and M\"uller also determine when $E(1,a)$ symplectically embeds into a scaling of $E(1,2)$ for all $a$.}.  In this paper, we apply Theorem~\ref{thm:mcduffhofer} to give new proofs of these results. We also show a surprising relationship between these staircases and a purely combinatorial phenomenon of independent interest concerning Ehrhart quasipolynomials of rational polytopes called ``period collapse." By exploiting this phenomenon, we prove that another infinite staircase appears when considering symplectic embeddings into scalings of $E(1,\frac{3}{2})$.    

The Ehrhart quasipolynomial of a $d$-dimensional rational polytope $\mathcal{P}$ is the counting function
\[
L_{\mathcal{P}}(t) \eqdef \# ( \mathcal{P} \cap \frac{1}{t} \Z^d).
\]
Let $\mathcal{D}(\mathcal{P})$ be the smallest $\mathcal{D} \in \Z_{> 0}$ such that the vertices of $\mathcal{D} \cdot \mathcal{P}$ are integral.  This is called the {\em denominator} of $\mathcal{P}$.  The Ehrhart quasipolynimal is always a degree $d$ polynomial in $t$ with periodic coefficients of period $\mathcal{D}(\mathcal{P})$.  The minimal period of $L_\mathcal{P}(t)$ is called the {\em period} of $\mathcal{P}$, and {\em period collapse} refers to any situation where the period of $\mathcal{P}$ is less than the denominator of $\mathcal{P}$.  There has been considerable interest in understanding when exactly period collapse occurs, see for example \cite{hm, mcallisterwoods, woods}.  

The link we establish between the staircases from \cite{ms} and \cite{fm} and period collapse comes from another theorem explaining when various families of triangles of symplectic interest have the same Ehrhart quasipolynomial.  Specifically, for positive real numbers $u$ and $v$, denote the triangle with vertices $(0,0), (0,u)$ and $(v,0)$ by $\mathcal{T}_{u,v}$, and call two triangles {\em Ehrhart equvialent} if they have the same Ehrhart quasipolynomial. For positive integers $k, l, p, q$ with $kp$ and $lq$ relatively prime, if $\mathcal{T}_{\frac{q}{kp},\frac{p}{lq}}$ and $\mathcal{T}_{\frac{1}{k},\frac{1}{l}}$ are Ehrhart equivalent we must have
\begin{equation}
\label{eqn:diophantine}
kp^2-(k+l+1)pq+lq^2 + 1 = 0.
\end{equation}
The equation \eqref{eqn:diophantine} comes from equating the linear terms of the corresponding Ehrhart quasipolynomials and its short derivation will be presented in \S\ref{sec:eqfds}. 

We will have reason to restrict our attention to pairs $(k,l)$ for which both $k$ and $l$ divide $k+l+1$. The only such values with $k \geq l$ are 
\begin{equation}
\label{eqn:kl_triplet} 
(k,l) \in \{(1,1), (2,1), (3,2) \},
\end{equation}
and unless explicitly stated otherwise, we will assume this holds. Our first theorem states that for these values, \eqref{eqn:diophantine} is the {\em only} obstruction to Ehrhart equivalence: 

\begin{theorem} 
\label{thm:periodcollapse}
Suppose $(k,l) \in \{(1,1), (2,1), (3,2)\}$ and assume that $kp$ and $lq$ are relatively prime.  Then
$\mathcal{T}_{\frac{q}{kp},\frac{p}{lq}}$ and $\mathcal{T}_{\frac{1}{k},\frac{1}{l}}$ are Ehrhart equivalent if and only if $(k,l,p,q)$ satisfies \eqref{eqn:diophantine}.
\end{theorem}

To explain the relevance of Theorem~\ref{thm:periodcollapse} to symplectic embeddings of ellipsoids, for positive real numbers $a,b$ and $t$, let 
\[
\N(a,b;t) = \#\{i : c_i(E(a,b)) \leq t\}.
\]
By Theorem~\ref{thm:mcduffhofer}, $\op{Int} E(a,b)\,\se\,E(c,d)$ if and only if $\N(a,b;t) \geq \N(c,d; t)$ for all $t$.  Now assume $c$ and $d$ are integers. Then $c_k(E(c,d))$ is always an integer, so it suffices to check $\N(a,b;t) \geq \N(c,d; t)$ for $t$ a positive integer. Since $\N(a,b;t) = L_{\T_{\frac{1}{a}, \frac{1}{b}}}(t)$ for such $t$, we have proven the following in view of Theorem~\ref{thm:mcduffhofer}:
\begin{lemma}
\label{clm:theclm}
Let $c$ and $d$ be integers.  Then $\op{Int} E(a,b) \se E(c,d)$ if and only if
\[
L_{\T_{\frac{1}{a}, \frac{1}{b}}}(t) \geq L_{\T_{\frac{1}{c}, \frac{1}{d}}}(t) \quad \forall t \in \mathbb{Z}_{> 0}.
\]
\end{lemma}
Note that by scaling, to determine when one rational ellipsoid symplectically embeds into another, it suffices to consider integral ellipsoids.

Now define the function 
\begin{equation}
\label{eqn:mcduffschlenkfunction}
c(a,b) \eqdef \op{inf}\lbrace \mu: E(1,a) \,\se\, E(\mu, b \mu) \rbrace.
\end{equation}
By scaling, the function $c(a,b)$ completely determines when one four-dimensional ellipsoid symplectically embeds into another.  We can determine the solutions of \eqref{eqn:diophantine}, and so Lemma~\ref{clm:theclm} can be combined with Theorem~\ref{thm:periodcollapse} to better understand the function $c(a,b)$.  Specifically, define the sequence $r(k,l)_n$ by $r(k,l)_0=1, r(k,l)_1=1$ and
\begin{align}
\label{eqn:2_1_recur_odd} r(k,l)_{2n+1} &= \frac{k+l+1}{k} r(k,l)_{2n} - r(k,l)_{2n-1}, \\
\label{eqn:2_1_recur_even} r(k,l)_{2n} &= \frac{k+l+1}{l}r(k,l)_{2n-1} - r(k,l)_{2n-2}.
\end{align}
For example, the $r(1,1)_n$ are the odd-index Fibonacci numbers familiar from the work of McDuff and Schlenk \cite{ms} and the $r(2,1)_n$ are related to the Pell numbers.  Set $r(k,l)_{-1}=1$.  In \S\ref{sec:classification}, we show that the solutions of \eqref{eqn:diophantine} are precisely the pairs $(p,q) = (r(k,l)_{2n \pm 1}, r(k,l)_{2n})$ for $n \ge 0$.  

Also define the sequences 
\begin{flalign*}
a(k,l)_{n}:=\left\{ \begin{aligned}\frac{kr(k,l)_{n+1}^2}{lr(k,l)_n^2} & \qquad\textrm{if }n\textrm{ is even},\\
\frac{lr(k,l)_{n+1}^2}{kr(k,l)_n^2} & \qquad\textrm{if }n\textrm{ is odd},
\end{aligned}
\right.
\end{flalign*}
and 
\[b(k,l)_{n}:=\frac{r(k,l)_{n+2}}{r(k,l)_n}.\]
Finally, define the positive real number
\[ \phi(k,l)=\frac{k}{l} \left(\frac{k+l+1+\sqrt{(k+l+1)^2-4kl}}{2k}\right)^2.\]

We always have
\[
a(k,l)_0 < b(k,l)_0 < a(k,l)_1 < b(k,l)_1 < \ldots < \phi(k,l).
\]
By combining Lemma~\ref{clm:theclm} with Theorem~\ref{thm:periodcollapse}, we can deduce the following  ``staircase" theorem concerning the function $c(a,b)$:
\begin{theorem}
\label{thm:setheorem}
Suppose $(k,l) \in \{(1,1), (2,1), (3,2)\}$.  For $a$ in the interval $[1,\phi(k,l)]$,
\[
c(a,\frac{k}{l}) = \begin{cases}
1 & \textrm{if }a \in [1, \frac{k}{l}], \\
\frac{a}{\sqrt{\frac{k}{l} a(k,l)_{n}}} & \textrm{if } a \in [a(k,l)_n,b(k,l)_n], \\
\sqrt{\frac{l}{k}a(k,l)_{n+1}} & \textrm{if } a \in [b(k,l)_n,a(k,l)_{n+1}].
\end{cases}
\]
\end{theorem}

Thus, for $(k,l) \in \{(1,1), (2,1), (3,2)\},$ the graph of $c(a,\frac{k}{l})$ begins with an infinite staircase, see Figure~\ref{fig:3_2}.  For $(k,l)=(1,1)$, Theorem~\ref{thm:setheorem} is a restatement of \cite[Thm. 1.1.i]{ms} and for $(k,l)=(2,1)$ it is a restatement of \cite[Thm. 1.1.i]{fm}, although as previously mentioned the proofs of \cite[Thm. 1.1.i]{ms} and \cite[Thm. 1.1.i]{fm} do not use Theorem~\ref{thm:mcduffhofer}.  For a survey of the methods from \cite{ms}, see \cite{mcduffsurvey}; the methods in \cite{fm} are similar.  

\begin{remark}
For $(k,l) \in \{(1,1), (2,1), (3,2)\},$ an argument using Theorem~\ref{thm:mcduffhofer} (and Lemma~\ref{clm:theclm}) shows that if 
\[
a \ge \frac{k}{l}\left(1+\frac{l+1}{k}\right)^2,
\]
then the only obstruction to symplectically embedding $E(1,a)$ into a scaling of $E(1,\frac{k}{l})$ is the volume. This is explained in the appendix.     
\end{remark}

The relevance of Theorem~\ref{thm:periodcollapse} to period collapse comes from the observation that if $k$ and $l$ are relatively prime then the period of $\mathcal{T}_{\frac{1}{k},\frac{1}{l}}$ is $kl$. In fact, for $(k,l) \in \{(1,1), (2,1)\}$ we can explain how to classify all such triangles for which this period collapse occurs. Specifically, we show:

\begin{theorem}
\label{thm:periodcollapse2}
Assume that $p$ and $q$ are relatively prime.  
\begin{enumerate} [(i)]
\item If $(k,l) \in \lbrace (1,1), (2,1)\rbrace$, then the Ehrhart quasipolynomial of $\mathcal{T}_{\frac{q}{kp},\frac{p}{lq}}$ has period $kl$ if and only if for some $n \ge 0$, 
\[(p,q)=(r(k,l)_{2n\pm 1},r(k,l)_{2n}) \quad \op{or} \quad (p,q)=(lr(k,l)_{2n},kr(k,l)_{2n\pm1}).\]
\item The Ehrhart quasipolynomial of $\mathcal{T}_{\frac{q}{3p},\frac{p}{2q}}$ has period $6$ if 
\[(p,q)=(r(3,2)_{2n\pm1},r(3,2)_{2n}) \quad \op{or} \quad (p,q)=(2r(3,2)_{2n},3r(2,1)_{2n\pm1}).\] 
\end{enumerate}
\end{theorem}
For example, Theorem~\ref{thm:periodcollapse2} gives examples of triangles with arbitrarily high denominator and period $1$, compare \cite[Ex. 2.1]{hm}.
 
One could ask whether the very strong condition that $k$ and $l$ both divide $k+l+1$ in Theorem~\ref{thm:periodcollapse} can be replaced by something weaker.  This would have implications for symplectic embeddings of ellipsoids along the lines of Theorem~\ref{thm:setheorem}.  Without some extra condition on $k$ and $l$, Theorem~\ref{thm:periodcollapse} certainly does not hold.  For example, for $(k,l)=(3,1)$ there are many examples of triangles $\mathcal{T}_{\frac{q}{kp},\frac{p}{lq}}$ that are not Ehrhart equivalent to $\mathcal{T}_{\frac{1}{k},\frac{1}{l}}$ even if $(k,l,p,q)$ satisfies \eqref{eqn:diophantine}.  Moreover, experimental evidence suggests that this condition on $k$ and $l$ cannot be replaced by any weaker condition in Theorem~\ref{thm:periodcollapse}.  It is also interesting to ask whether other infinite staircases appear in the graph of $c(a,b)$.  Work of Frenkel and Schlenk  \cite{fs} implies that $c(a,4)$ is equal to the volume obstruction except on finitely many intervals for which it is linear, and it is suspected by both the authors and Schlenk \cite{s} that the graph of $c(a,k)$ never contains an infinite staircase for even integer $k \ge 4$.  A simple characterization of period collapse for the family $\mathcal{T}_{u,v}$ along the lines of Theorem~\ref{thm:periodcollapse2} is also open.

\begin{figure}[!ht, p]
\centering
\includegraphics[width=\textwidth]{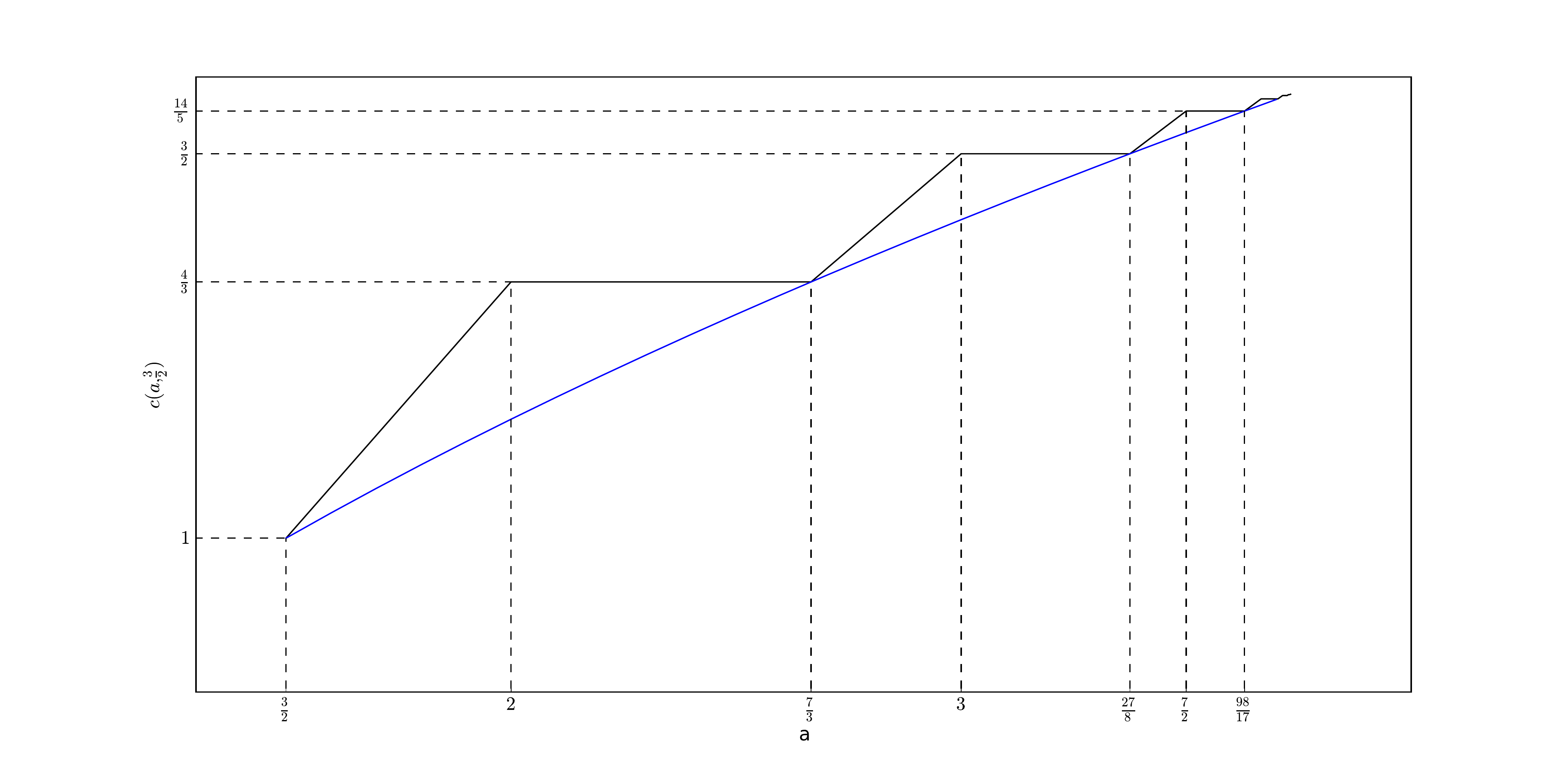}
\caption{Staircase for $(k,l) = (3,2)$}
\label{fig:3_2}
\end{figure}

\paragraph{Acknowledgements} We are indebted to Matthias Beck, Michael Hutchings, Dusa McDuff and Felix Schlenk for helpful conversations and for their encouragement throughout the project, and we thank Michael Hutchings, Felix Schlenk, and David Frenkel for patiently reading over earlier drafts of this work.  This project grew out of conversations with Daniel Bragg, Peter Cheng, Caitlin Stanton, and Olga Vasileva during the $2012$ UC Berkeley Geometry, Topology, and Operator Algebras RTG Summer Research Program for Undergraduates.  This program was funded by the National Science Foundation.  

\section{Preliminaries}\label{sec:preliminaries}
We begin by developing the combinatorial machinery that will be used in the proof of Theorem~\ref{thm:periodcollapse}. 
\subsection{Ehrhart quasipolynomials and Fourier-Dedekind sums}\label{sec:eqfds}
Given a triangle of the form 
\[
\mathcal{T} = \left\{ (x,y) \in \R^2 : x \geq \frac{a}{d}, y \geq \frac{b}{d}, ex+fy \leq r \right\},
\]
one can use generating functions to compute the Ehrhart quasipolynomial of $\mathcal{T}$. This is explained in \cite[\S 2]{br}. We will only need to consider the special case where $a=b=0$, $e=kp^2$, $f=lq^2$, and $r = pq$, for $p, q, k, l$ positive integers with $kp^2$ and $lq^2$ relatively prime.  In this case, \cite[Thm. 2.10]{br} gives:
\begin{equation}
\label{eqn:keyformula}
\begin{aligned}
L_\mathcal{T}(t) &= \frac{1}{2kl}t^2  + \frac{1}{2}\left(\frac{q}{kp} + \frac{p}{lq} + \frac{1}{klpq}\right)t \\
&+\frac{1}{4}\left(1+\frac{1}{kp^2}+\frac{1}{lq^2}\right) + \frac{1}{12}\left(\frac{kp^2}{lq^2} + \frac{lq^2}{kp^2} + \frac{1}{klp^2q^2}\right)\\
&+s_{-tpq}(lq^2,1;kp^2) + s_{-tpq}(kp^2,1;lq^2),
\end{aligned}
\end{equation}
where $s_n$ denotes the {\em Fourier-Dedekind sum}
\[
s_n(a_1,a_2; b) = \frac{1}{b} \sum_{k=1}^{b-1} \frac{\xi_b^{kn}}{(1-\xi_b^{a_1k}) (1-\xi_b^{a_2k})}.
\]
Here and in the following sections, $\xi_b$ denotes the primitive $b^{th}$ root of unity $\xi_b = e^{\frac{2\pi i}{b}}$.

Since Theorem~\ref{thm:periodcollapse} is an equivalence of quasipolynomials, we must have equality in the coefficients of each power of $t$. Equating the linear terms given by \eqref{eqn:keyformula} yields the Diophantine equation \eqref{eqn:diophantine}, and thus gives the ``only if'' direction of Theorem~\ref{thm:periodcollapse}. In \S\ref{sec:classification} we completely classify the solutions of \eqref{eqn:diophantine} when $(k,l)$ satisfies \eqref{eqn:kl_triplet}. 	

\subsection{Convolutions}
In view of \eqref{eqn:keyformula}, the hard part of Theorem~\ref{thm:periodcollapse} is evaluating the expression
\begin{equation}\label{eqn:hardpart}
s_{-tpq}(kp^2,1;lq^2)+s_{-tpq}(lq^2,1;kp^2).
\end{equation}
So for $(k,l,p,q)$ satisfying \eqref{eqn:diophantine}, consider the sum
\[
s_{-tpq}(kp^2,1;lq^2) = \frac{1}{lq^2} \sum_{j=1}^{lq^2-1} \frac{\xi_{lq^2}^{-tjpq}}{(1-\xi_{lq^2}^{jkp^2})(1-\xi_{lq^2}^j)}.
\]
Writing $j=ilq+u$ for $0 \leq i < q$, $0 \leq u < lq$ gives 
\begin{equation}
\label{eqn:convolution1}
\begin{aligned}
\frac{1}{lq^2} &\sum_{i=1}^{q-1} \frac{\xi_{lq^2}^{-tpq^2il}}{(1-\xi_{lq^2}^{kp^2 ilq})(1-\xi_{lq^2}^{ilq})} + \frac{1}{lq^2} \sum_{\u=1}^{lq-1} \left(\xi_{lq^2}^{-tpq\u} \sum_{i=0}^{q-1} \frac{1}{(1-\xi_{lq^2}^{kp^2(ilq+\u)})(1-\xi_{lq^2}^{ilq+\u})}\right) \\
= & \frac{1}{lq^2} \sum_{i=1}^{q-1} \frac{1}{(1-\xi_{q}^{-i})(1-\xi_{q}^i)} + \frac{1}{lq^2} \sum_{\u=1}^{lq-1}\left(  \xi_{lq^2}^{-tpq\u} \sum_{i=0}^{q-1} \frac{1}{(1-\xi_{lq^2}^{\u kp^2-ilq})(1-\xi_{lq^2}^{\u+ilq})}\right).
\end{aligned}
\end{equation}
%where the last line of \eqref{eqn:convolution1} follows by \eqref{eqn:diophantine}.
The sum 
\[ s_{-tpq}(lq^2,1;kp^2) = \frac{1}{kp^2} \sum_{j=1}^{kp^2-1} \frac{\xi_{kp^2}^{-tjpq}}{(1-\xi_{kp^2}^{jlq^2})(1-\xi_{kp^2}^j)} \]
can be similarly rewritten %by writing $j=ikp+\u$ and applying \eqref{eqn:diophantine}
to obtain
\begin{equation}
\label{eqn:convolution2}
\frac{1}{kp^2} \sum_{i=1}^{p-1} \frac{1}{(1-\xi_{p}^{-i})(1-\xi_{p}^i)} + \frac{1}{kp^2} \sum_{\u=1}^{kp-1}\left(  \xi_{kp^2}^{-tpq\u} \sum_{i=0}^{p-1} \frac{1}{(1-\xi_{kp^2}^{\u lq^2-ikp})(1-\xi_{kp^2}^{\u+ikp})}\right).
\end{equation}
The inner sums of \eqref{eqn:convolution1} and \eqref{eqn:convolution2} are {\em convolutions} that can be evaluated explicitly using the Fourier transform. %This will be explained in \S\ref{sec:ft}.

\subsection{Fourier transform}
\label{sec:ft}

To evaluate the convolutions in \eqref{eqn:convolution1} and \eqref{eqn:convolution2}, we will apply the following general lemma:
\begin{lemma}
\label{lem:keylemma}
Let $a_1,a_2,$ $b$ and $c$ be integers such that $b$ divides neither $a_1$ nor $a_2$.  Then
\begin{equation}\label{eqn:thelemma}
\frac{1}{c} \sum_{k=0}^{c-1} \frac{1}{(1-\xi_{bc}^{a_1+kb})(1-\xi_{bc}^{a_2-kb})} = 
\frac{\gamma}{(1-\xi_{bc}^{a_1c})(1-\xi_{bc}^{a_2c})},
\end{equation}
where 
\[
\gamma = \begin{cases}
\frac{1-\xi_{bc}^{(a_1+a_2)c}}{1-\xi_{bc}^{a_1 + a_2}} & \mathrm{if\ } bc \not| a_1 + a_2, \\
c & \mathrm{if\ }bc | a_1 + a_2.
\end{cases}
\]
\end{lemma}

We will be most interested in the lemma when in addition $b$ divides $a_1 + a_2$ but $bc$ does not divide $a_1+a_2$.  In this case, Lemma~\ref{lem:keylemma} implies that the sum in \eqref{eqn:thelemma} is equal to $0$.
\begin{proof}
Our proof is given in three steps.
  
{\em Step 1.\/}  This step summarizes the inputs from finite Fourier analysis.

If $f$ is a function with period $b$, recall that its
{\em Fourier transform} is the function
\[
\hat{f}(n) = \frac{1}{b} \sum_{k=0}^{b-1} f(k) \xi_b^{-kn}.
\]
The convolution of two periodic functions $f,g$ with period $b$ is given by
\[
(f*g)(n) = \sum_{m=0}^{b-1}f(n-m)g(m).
\]
A version of the {\em convolution theorem} \cite[Thm. 7.10]{br} for the Fourier transform says that
\begin{equation}
\label{eqn:convolutiontheorem}
(f*g)(n) = b \sum_{k=0}^{b-1} \hat{f}(k) \hat{g}(k) \xi_b^{kn}.
\end{equation}

{\em Step 2.\/}  We can compute the Fourier transform of the family of functions that are relevant to the proof of Lemma~\ref{lem:keylemma} explicitly:
\begin{lemma}\label{lemma:hat}
Fix positive integers $a,b$ and $c$ such that $b$ does not divide $a$.
Let $f_a$ be the periodic function of period $c$ given by 
\[
f_a(n) \eqdef \frac{1}{1-\xi_{bc}^{a+bn}}.
\]
Then for integers $0 \leq n \leq c-1$,
\[
\hat{f}_a(n) = \frac{\xi_{bc}^{an}}{1-\xi_{bc}^{ac}}.
\]
\end{lemma}
\begin{proof}
For notational simplicity, throughout this proof let $\xi := \xi_{bc}$. 
For $n \geq 0$, we have
\begin{align*}
\hat{f}_a(n) &= \frac{1}{c} \sum_{k=0}^{c-1} \frac{\xi^{-knb}}{1-\xi^{a+kb}} \\
 &= \frac{\xi^{an}}{c} \sum_{k=0}^{c-1} \frac{\xi^{-(a+kb)n}}{1-\xi^{a+kb}} \\
 &= \frac{\xi^{an}}{c} \sum_{k=0}^{c-1} \left( \frac{1}{1-\xi^{a+kb}} - \frac{1 - \xi^{-(a+kb)n}}{1-\xi^{a+kb}} \right) \\
 &= \frac{\xi^{an}}{c} \sum_{k=0}^{c-1} \left(\frac{1}{1-\xi^{a+kb}} + \sum_{m=1}^{n} \xi^{-(a+kb)m} \right).
\end{align*}
We can break the last line up into two sums and interchange the order of summation in the last sum to get
\begin{equation}
\label{eqn:interchanged}
\frac{\xi^{an}}{c} \sum_{k=0}^{c-1} \frac{1}{1-\xi^{a+kb}} + \frac{1}{c} \sum_{m=1}^{n} \left( \xi^{a(n-m)} \sum_{k=0}^{c-1} \xi^{-kbm} \right).
\end{equation}
The innermost sum on the right hand side of \eqref{eqn:interchanged} is $0$ if $c$ does not divide $m$. Since $m \leq n$, when $0 \leq n \leq c-1$ the sum always vanishes and
\begin{equation}
\label{eqn:fouriertransformequation}
\hat{f}_a(n)=\frac{\xi^{an}}{c} \sum_{k=0}^{c-1} \frac{1}{1-\xi^{a+kb}}.
\end{equation}
To simplify the summation, let $z_k=\frac{1}{1-\xi^{a+kb}}$ and note that the $z_k$ are the roots of the degree-$c$ polynomial $(z-1)^c = \xi^{ac}z^c$.
Hence
\begin{equation}
\label{eqn:polynomialformulation} 
(z-1)^c - \xi^{ac}z^c = (1-\xi^{ac}) \prod_{k=0}^{c-1}(z-z_k). 
\end{equation}
Equating the coefficient of $z^{c-1}$ on each side of \eqref{eqn:polynomialformulation} gives
\begin{equation}
\label{eqn:keyequation}
\sum_{k=0}^{c-1} \frac{1}{1-\xi^{a+kb}} = \frac{c}{1-\xi^{ac}},
\end{equation}
and Lemma~\ref{lemma:hat} now follows by combining \eqref{eqn:fouriertransformequation} and \eqref{eqn:keyequation}. 
\end{proof}
{\em Step 3.}  The sum in \eqref{eqn:thelemma} is $(f_{a_1}*f_{a_2})(0)$, so by \eqref{eqn:convolutiontheorem} and Lemma~\ref{lemma:hat}, \begin{align*}
(f_{a_1} * f_{a_2})(0) &= c  \sum_{k=0}^{c-1} \left(\frac{\xi^{a_1 k}}{1-\xi^{a_1c}}\right) \left(\frac{\xi^{a_2 k}}{1-\xi^{a_2c}}\right) \\
&= \frac{c}{(1-\xi^{a_1c})(1-\xi^{a_2c})} \sum_{k=0}^{c-1} \xi^{(a_1+a_2)k}.
\end{align*}
The sum in the last line evaluates to $c$ if $bc | a_1 + a_2$, and otherwise we have 
\[
\sum_{k=0}^{c-1} \xi^{(a_1 + a_2)k} = \frac{1 - \xi^{(a_1 + a_2)c}}{1-\xi^{a_1 + a_2}}.
\]
This completes the proof. 
\end{proof}

\subsection{Reciprocity}
\label{sec:reciprocity}
Certain expressions with Fourier-Dedekind sums can be evaluated directly. The proof of Theorem~\ref{thm:periodcollapse} uses the following result from \cite[Thm.  8.8]{br}:

\begin{lemma} (Rademacher Reciprocity)
\label{lem:reciprocity}
Let $n=1,2,\ldots,(a+b+c)-1$.  Then
\begin{align*}
s_n(a,b;c) &+ s_n(c,a;b)+s_n(b,c;a) =\\
& - \frac{n^2}{2abc} + \frac{n}{2}\left(\frac{1}{ab}+\frac{1}{ca} + \frac{1}{bc}\right) - \frac{1}{12}\left(\frac{3}{a}+\frac{3}{b}+\frac{3}{c}+\frac{a}{bc}+\frac{b}{ca}+\frac{c}{ab}\right).
\end{align*}
\end{lemma}

There is another reciprocity statement for $n=0$:
\begin{lemma} \cite[Cor. 8.7]{br}:
\label{lem:reciprocity2}
\begin{align*}
s_0(a,b;c) &+ s_0(c,a;b)+s_0(b,c;a) =\\
& 1 - \frac{1}{12}\left(\frac{3}{a}+\frac{3}{b}+\frac{3}{c}+\frac{a}{bc}+\frac{b}{ca}+\frac{c}{ab}\right).
\end{align*}
\end{lemma}

\section{Proof of Theorem~\ref{thm:periodcollapse}}
We will now apply the machinery from the previous section to prove Theorem~\ref{thm:periodcollapse}. 
We showed the ``only if'' direction in \S\ref{sec:eqfds}, so it remains to show the ``if'' direction. Assume $k$ and $l$ both divide $k+l+1$, that $kp^2$ and $lq^2$ are relatively prime, and that $(k,l,p,q)$ satisfies \eqref{eqn:diophantine}. The proof that $\mathcal{T}_{\frac{q}{kp}, \frac{p}{lq}}(t) = \mathcal{T}_{\frac{1}{k}, \frac{1}{l}}(t)$ for all positive integers $t$ follows in four steps.

{\em Step 1.} Both $\mathcal{T}_{\frac{q}{kp}, \frac{p}{lq}}$ and $\mathcal{T}_{\frac{1}{k}, \frac{1}{l}}$ are quadratic quasipolynomials in $t$. By \eqref{eqn:keyformula} both have the same coefficient of $t^2$, and by \eqref{eqn:diophantine} both have the same coefficient of $t$. It remains to show that they have the same constant term.

{\em Step 2.}  To evaluate the relevant Fourier-Dedekind sums, the following elementary fact will be useful:    
\begin{lemma}
\label{lem:usefulfact}
$q$ is relatively prime to $\frac{k+l+1}{l}$ and $p$ is relatively prime to $\frac{k+l+1}{k}$.
\end{lemma}
\begin{proof}  Since $(k,l)$ satisfies \eqref{eqn:kl_triplet} we can argue case by case.  If $(k,l)=(1,1)$, then Lemma~\ref{lem:usefulfact} follows by reducing \eqref{eqn:diophantine} $\mathrm{mod\ }3$.  If $(k,l)=(2,1)$, then reducing \eqref{eqn:diophantine} mod $8$ shows that $p$ and $\frac{k+l+1}{k}$ are relatively prime, and reducing \eqref{eqn:diophantine} mod $4$ shows that $q$ and $\frac{k+l+1}{l}$ are relatively prime.  Finally, if $(k,l)=(3,2)$ then reducing \eqref{eqn:diophantine} mod $3$ shows that $q$ and $\frac{k+l+1}{l}$ are relatively prime, and reducing \eqref{eqn:diophantine} mod $2$ shows that $p$ and $\frac{k+l+1}{k}$ are relatively prime.
\end{proof}

{\em Step 3.} We now begin the computation of  the Fourier-Dedekind sums. By \eqref{eqn:diophantine}, \eqref{eqn:convolution1}, and \eqref{eqn:convolution2} we have 
\[s_{-tpq}(kp^2,1;lq^2)+s_{-tpq}(lq^2,1;kp^2)=\]
\begin{equation}
\label{eqn:convolution}
\begin{aligned}
\frac{1}{lq^2} \sum_{i=1}^{q-1} \frac{1}{(1-\xi_{q}^{-i})(1-\xi_{q}^i)} + \frac{1}{lq^2} \sum_{\u=1}^{lq-1}\left(  \xi_{lq^2}^{-tpq\u} \sum_{i=0}^{q-1} \frac{1}{(1-\xi_{lq	^2}^{\u kp^2-ilq})(1-\xi_{lq^2}^{\u+ilq})}\right) + \\
\frac{1}{kp^2} \sum_{i=1}^{p-1} \frac{1}{(1-\xi_{p}^{-i})(1-\xi_{p}^i)} + \frac{1}{kp^2} \sum_{\u=1}^{kp-1}\left(  \xi_{kp^2}^{-tpq\u} \sum_{i=0}^{p-1} \frac{1}{(1-\xi_{kp^2}^{\u lq^2-ikp})(1-\xi_{kp^2}^{\u+ikp})}\right).
\end{aligned}
\end{equation}

By \eqref{eqn:diophantine} we always have $q | \u(kp^2 + 1)$, and by Lemma~\ref{lem:usefulfact}, $lq^2 | \u(kp^2 + 1)$ if and only if $q | u$. So applying Lemma~\ref{lem:keylemma} with $b=lq, c =q, a_1 =  u$ and $a_2 = \u kp^2$ gives
\begin{equation} \label{eqn:keyeqn1}
\sum_{i=0}^{q-1} \frac{1}{(1-\xi_{lq^2}^{\u kp^2-liq})(1-\xi_{lq^2}^{\u+liq})} =\begin{cases}
0 & \mathrm{if\ }q \not| u \\
\frac{q^2}{(1-\xi_{lq}^{kp^2\u})(1-\xi_{lq}^{\u})} & \mathrm{if\ }q | u
\end{cases}
\end{equation}

An identical argument gives
\begin{equation}\label{eqn:keyeqn2} %\label{eqn:keyeqn3}
\sum_{i=0}^{p-1} \frac{1}{(1-\xi_{kp^2}^{\u lq^2-kip})(1-\xi_{kp^2}^{\u+kip})} =\begin{cases}
0 & \mathrm{if\ }p \not| u \\
\frac{p^2}{(1-\xi_{kp}^{lq^2\u})(1-\xi_{kp}^{\u})} & \mathrm{if\ }p | u
\end{cases}
\end{equation}
Now $kp^2 \equiv -1 \pmod{l}$ and $lq^2 \equiv -1 \pmod{k}$ by ~\eqref{eqn:diophantine}, so combining \eqref{eqn:keyeqn1} and \eqref{eqn:keyeqn2}
gives 
\[s_{-tpq}(kp^2,1;lq^2)+s_{-tpq}(lq^2,1;kp^2)=\]
\begin{equation}
\label{eqn:finalshot}
\begin{aligned}
\frac{1}{lq^2} \sum_{i=1}^{q-1} \frac{1}{(1-\xi_{q}^{-i})(1-\xi_{q}^i)}+\frac{1}{kp^2} \sum_{i=1}^{p-1} \frac{1}{(1-\xi_{p}^{-i})(1-\xi_{p}^i)}+ \\
\frac{1}{l} \sum_{i=1}^{l-1}\frac{\xi_{l}^{-tip}}{(1-\xi_{l}^{i})(1-\xi_{l}^{-i})}+\frac{1}{k} \sum_{i=1}^{k-1}\frac{\xi_{k}^{-tiq}}{(1-\xi_{k}^{i})(1-\xi_{k}^{-i})}. 
\end{aligned}
\end{equation}

{\em Step 4.} By \eqref{eqn:keyformula}, we must show 
\begin{equation}
\label{eqn:const_terms}
\begin{aligned}
\frac{1}{4} & \left(1+\frac{1}{kp^2}+\frac{1}{lq^2}\right) + \frac{1}{12}\left(\frac{kp^2}{lq^2} + \frac{lq^2}{kp^2} + \frac{1}{klp^2q^2}\right) \\
&+ s_{tpq}(lq^2,1;kp^2) + s_{tpq}(kp^2,1;lq^2)  \\
& = \frac{1}{4}\left(1 + \frac{1}{k} + \frac{1}{l}\right) + \frac{1}{12} \left(\frac{k}{l} + \frac{l}{k} + \frac{1}{kl}\right) + s_{t}(l,1;k) + s_{t}(k,1;l).
\end{aligned}
\end{equation}
for all $t \leq 0$. The right hand side of~\eqref{eqn:const_terms} is periodic in $t$ with period $kl$, and by \eqref{eqn:finalshot}, the left hand side is as well. For $(k,l) = (3,2)$, by \eqref{eqn:finalshot} the right hand side is equal at $t = 1$ and $t = 5$, and is equal at $t=2$ and $t=4$. Thus, when $(k,l) = (1,1)$ we can assume $t=0$, when $(k,l) = (2,1)$ we can assume $t = 0$ or $1$, and when $(k,l) = (3,2)$ we can assume $0 \leq t \leq 3$. 

When $t = 0$, we can apply Lemma~\ref{lem:reciprocity2} to both sides of \eqref{eqn:const_terms} to get the desired equality.   For other $t$, we can apply Rademacher reciprocity to evaluate $s_{tpq}(kp^2,1;lq^2) + s_{tpq}(lq^2,1;kp^2)$ as long as $0 < tpq < kp^2 + lq^2$. This holds for all $p,q$ when $0 < t < 2 \sqrt{kl}$, and we can always assume $t$ lies in this range by the previous paragraph. For these $t$
Lemma~\ref{lem:reciprocity} gives
\begin{equation}
\label{eqn:thekeyequation}
\begin{aligned}
s_{tpq}(kp^2,1;lq^2) &+s_{tpq}(lq^2,1;kp^2)  \\
&=  -\frac{1}{12}\left(\frac{3}{kp^2}+\frac{3}{lq^2}+3+\frac{kp^2}{lq^2}+\frac{lq^2}{kp^2}+\frac{1}{klp^2q^2}\right) \\
& -\frac{t^2}{2kl} + \frac{t}{2}\left(\frac{1}{klpq}+\frac{q}{kp}+\frac{p}{lq}\right).
\end{aligned}
\end{equation}

Since \eqref{eqn:thekeyequation} also holds for $(p,q)=(1,1)$, and because
\begin{equation}
\label{eqn:likelinearthing}
\frac{pq}{2}\left(\frac{1}{klp^2q^2}+\frac{1}{kp^2}+\frac{1}{lq^2}\right)=\frac{1}{2}\left(\frac{1}{kl}+\frac{1}{k}+\frac{1}{l}\right)
\end{equation}
by \eqref{eqn:diophantine}, Theorem~\ref{thm:periodcollapse} follows in this case as well by \eqref{eqn:keyformula}.

\section{Classification of solutions to \eqref{eqn:diophantine}}\label{sec:classification}
In this section we prove the following:
\begin{proposition}
\label{prop:classification}
Suppose $(k,l) \in \{(1,1), (2,1), (3,2)\}$. Then $(p,q)$ is a solution to~\eqref{eqn:diophantine} if and only if $(p,q) = (r(k,l)_{2n \pm 1}, r(k,l)_{2n})$ for some $n$.
\end{proposition}

\begin{remark}
To prove Theorem~\ref{thm:periodcollapse}, we only need the ``if" direction of Proposition~\ref{prop:classification}.  When $(k,l) \in \lbrace (1,1), (2,1) \rbrace$, the ``only if" direction will be used in the proof of Theorem~\ref{thm:periodcollapse2}.  We include the $(k,l)=(3,2)$ case here for completeness in view of Theorem~\ref{thm:periodcollapse}.
\end{remark} 

\begin{proof}
Fix $(k,l) \in \{(1,1), (2,1), (3,2)\}$, and consider the pair of congruence relations
\begin{equation}\label{eqn:dioph_pair}
k a^2 \equiv -1 \pmod{lb}, \quad \quad lb^2 \equiv -1 \pmod{ka}.
\end{equation}
Since $k$ and $l$ both divide $k+l+1$, if $(p,q)$ satisfies~\eqref{eqn:diophantine} then $(a,b) = (p,q)$ is a solution to~\eqref{eqn:dioph_pair}. We will show that the converse holds, so it suffices to classify the solutions of \eqref{eqn:dioph_pair}, and we will then show this is precisely the set of $(r(k,l)_{2n \pm 1}, r(k,l)_{2n})$. 

We first solve \eqref{eqn:dioph_pair}.  The key observation  is that if $(p,q)$ is a solution of~\eqref{eqn:dioph_pair}, then \[p' := \frac{lq^2 + 1}{kp}, \quad \quad q' := \frac{kp^2 + 1}{lq}\]
are integers and $(p',q)$ and $(p,q')$ are also solutions to~\eqref{eqn:dioph_pair}.  Motivated by this, we define the involutions
\[\sigma: (p,q) \to (\frac{lq^2 + 1}{kp}, q), \quad \quad \tau: (p,q) \to (p, \frac{kp^2 + 1}{lq}).\]

We claim that if $(p,q) \ne (1,1)$ then either $\sigma$ or $\tau$ decreases a coordinate.  Suppose that $p \leq p'$ and $q \leq q'$. Then $|kp^2 - lq^2| \leq 1$. If $kp^2 = lq^2$ then $(k,l,p,q) = (1,1,1,1)$, if $kp^2 = lq^2 + 1$ then $lq | 2$ so $(l,q) \in \{(1,1), (1,2), (2,1)\}$, and if $kp^2 = lq^2 - 1$ then $(k,p) \in \{(1,1), (1,2), (2,1)\}$. By examining each of these cases separately we see that if $k \geq l$, $p' \geq p$, $q' \geq q$, then
\[(k,l,p,q) \in \lbrace (1,1,1,1), (2,1,1,1), (3,2,1,1), (5,1,1,2) \rbrace.\]
In particular, if we assume in addition that $(k,l)$ satisfies \eqref{eqn:kl_triplet}, then $(p,q)=(1,1)$.

Now define the sequence $s(k,l)_n$ by $s(k,l)_0 = s(k,l)_1 = 1$, 
\begin{equation}\label{eqn:sn_recurrence}
s_{2n+1} = \frac{l s(k,l)_{2n}^2+1}{ks(k,l)_{2n-1}}, \quad \quad s(k,l)_{2n} = \frac{ks(k,l)_{2n-1}^2 + 1}{l s(k,l)_{2n-2}}.
\end{equation}
If $(p,q)$ satisfies~\eqref{eqn:dioph_pair} then $(p,q) = (s(k,l)_{2n \pm 1}, s(k,l)_{2n})$ for some $n$.  This follows by induction after applying either $\sigma$ or $\tau$.  Another induction using~\eqref{eqn:sn_recurrence} shows that $(s(k,l)_{2n \pm 1}, s(k,l)_{2n})$ satisfies~\eqref{eqn:diophantine}.  Thus, the solutions of \eqref{eqn:diophantine} and \eqref{eqn:dioph_pair} are the same.

To show the $r_n = s_n$ for all $n$, we induct using \eqref{eqn:2_1_recur_odd} and \eqref{eqn:2_1_recur_even} to get
\begin{align}
\label{eqn:r_dioph_pair_odd} k r_{2n+1}^2  - (k+l+1) r_{2n+1} r_{2n} + l r_{2n}^2 &= -1,\\
\label{eqn:r_dioph_pair_even} k r_{2n-1}^2 - (k+l+1) r_{2n-1} r_{2n} + l r_{2n}^2 &= -1.
\end{align}
 We can then apply a final induction using the recurrence relations \eqref{eqn:sn_recurrence}, \eqref{eqn:2_1_recur_even} and \eqref{eqn:2_1_recur_odd}.

\end{proof}

\section{Proof of Theorem~\ref{thm:setheorem}}

There are several basic properties of $c(a,\frac{k}{l})$ that significantly simplify the proof of Theorem~\ref{thm:setheorem}:

\begin{lemma} \label{lem:axioms} Fix $k$ and $l$.  Then the function $c(a,\frac{k}{l})$ satisfies:
\begin{enumerate} [(i)]
\item (Continuity) $c(a,\frac{k}{l})$ is a continuous function of $a$.
\item (Monotonicity) $c(a,\frac{k}{l})$ is a monotonically nondecreasing function of $a$.
\item (Subscaling) $c(\lambda a, \frac{k}{l}) \le \lambda c(a,\frac{k}{l})$ when $\lambda > 1$.
\end{enumerate}
\end{lemma}
\begin{proof}
To prove statement (i), note that for any $\epsilon > 0$, if $a_i$ is close enough to $a$ then $E(1,a_i) \subset (1+\epsilon)E(1,a)$ and $E(1,a) \subset (1+\epsilon) E(1,a_i)$.  Statement (i) then follows from the observation that if $E(1,x) \se E(c,c\frac{k}{l})$ then $(1+\epsilon)E(1,x) \se (1+\epsilon)E(c,c\frac{k}{l})$.  Statement (ii) follows from the fact that if $x \le y$ then $E(1,x) \subset E(1,y)$.  Statement (iii) follows because we have
\[ E(1,a) \subset \sqrt{\lambda} E(1,a)
\]
for any $\lambda > 1$, and we know that
\[ 
\sqrt{\lambda} E(c,c\frac{k}{l}) = E(\lambda c, \lambda c \frac{k}{l}).
\]
\end{proof}

Our strategy for proving Theorem~\ref{thm:setheorem} is to calculate $c(a(k,l)_n,\frac{k}{l})$, bound $c(b(k,l)_n,\frac{k}{l})$ from below, and apply Lemma~\ref{lem:axioms}.  

\subsection{Calculating $c(a(k,l)_n, \frac{k}{l})$}
We first claim that $c(a(k,l)_n,\frac{k}{l})$ is always equal to the volume obstruction.
%This is where  Theorem~\ref{thm:periodcollapse} is relevant to the proof of Theorem~\ref{thm:setheorem}.
To simplify the notation, we  now let $a_n, b_n,$ and $r_n$ denote $a(k,l)_n, b(k,l)_n,$ and $r(k,l)_n$ for fixed $(k,l)$ satisying \eqref{eqn:kl_triplet}.  

By definition, $\frac{a_{2n}l}{k} = \frac{r_{2n+1}^2}{r_{2n}^2}$ and $\frac{a_{2n+1}l}{k}=\frac{l^2r_{2n+2}^2}{k^2r_{2n+1}^2}$.  To show
\begin{align}
\label{eqn:c_a_2n} c(a_{2n}, \frac{k}{l}) &= \sqrt{\frac{a_{2n}l}{k}} = \frac{r_{2n+1}}{r_{2n}}, \\
\label{eqn:c_a_2n_1} c(a_{2n+1}, \frac{k}{l}) &= \sqrt{\frac{a_{2n+1}l}{k}}  = \frac{l r_{2n+2}}{kr_{2n+1}},
\end{align}
 it suffices by Lemma~\ref{clm:theclm} to show that
\begin{equation}\label{eqn:klpqt}
L_{\T_{\frac{q}{kp}, \frac{p}{lq}}}(t) \geq L_{\T_{\frac{1}{k}, \frac{1}{l}}}(t)
\end{equation}
when $(p,q) = (r_{2n\pm1}, r_{2n}).$ 

By induction, \eqref{eqn:2_1_recur_odd} and \eqref{eqn:2_1_recur_even} show
that that $r_{2n+1}$ and $r_{2n}$ are relatively prime.  Since $kl | k + l + 1$ for $(k,l) \in \{(1,1), (2,1), (3,2)\}$, induction also shows that $k \not| r_{2n}$, $l \not| r_{2n+1}$.  Then \eqref{eqn:klpqt} follows from \eqref{eqn:r_dioph_pair_odd}, \eqref{eqn:r_dioph_pair_even} and Theorem~\ref{thm:periodcollapse}.

\subsection{Calculating $c(b(k,l)_n, \frac{k}{l})$}
%Continue to let $a_n, b_n,$ and $r_n$ denote $a(k,l)_n, b(k,l)_n,$ and $r(k,l)_n$ for fixed $(k,l)$ satisying \eqref{eqn:kl_triplet}.  
By Lemma~\ref{lem:axioms} and \eqref{eqn:klpqt}, to prove Theorem~\ref{thm:setheorem} it remains to show that 
\[
c(b_n, \frac{k}{l}) \ge \sqrt{\frac{la_{n+1}}{k}} = \begin{cases}
\frac{r_{n+2}}{r_{n+1}} & n\ \mathrm{odd}, \\ 
\frac{l}{k} \frac{r_{n+2}}{r_{n+1}} & n\ \mathrm{even.}
\end{cases}.
\]
We will show that for the index
\begin{equation}\label{eqn:fn}
f_n := \frac{r_{n+2}r_{n} + r_{n+2} + r_{n} - 1}{2},
\end{equation}
we have
\begin{align}
\label{eqn:bn_idx} c_{f_n}(E(1,b_n)) &= r_{n+2}, \\
\label{eqn:beta_idx} c_{f_n}(E(1, \frac{k}{l}))&= \begin{cases}
r_{n+1} & n\ \mathrm{odd}, \\
\frac{k}{l}r_{n+1} & n\ \mathrm{even}.
\end{cases}
\end{align}

We begin with the proof of~\eqref{eqn:bn_idx}. We have
\begin{align*}
\max_m \{m : c_m(E(1,b_n)) \leq r_{n+2}\} &= -1 + \sum_{i=0}^{r_{n+2}} \left(\left\lfloor\frac{i}{b_n}\right\rfloor + 1 \right) \\
&= r_{n+2} + r_n + \sum_{i=0}^{r_{n+2}-1} \left\lfloor \frac{i r_n}{r_{n+2}} \right\rfloor \\
&= \frac{(r_{n+2}+1)(r_n+1)}{2},
\end{align*}
where the last line follows from the well-known identity
\[
\sum_{i=0}^{q-1} \left\lfloor \frac{ip}{q}\right\rfloor = \frac{(p-1)(q-1)}{2}
\]
for $(p,q)=1$. The fact that $\mathrm{gcd}(r_{n+2}, r_n)=1$ follows from an induction using~\eqref{eqn:2_1_recur_odd} and~\eqref{eqn:2_1_recur_even}.
Since
\[
\#\{m : c_m(E(1,b_n)) = r_{n+2}\} = 2,
\]
we have that $c_{f_n}(E(1,b_n)) = c_{f_{n+1}}(E(1,b_n))$, and \eqref{eqn:bn_idx} follows. 

We next prove \eqref{eqn:beta_idx}. We have that
\begin{equation} \label{eqn:fn_even_simple}
\begin{aligned}
f_{2n} &= \frac{r_{2n+2} r_{2n} + r_{2n+2} + r_{2n} - 1}{2} \\
&= \frac{1}{2}((\frac{k+l+1}{l} r_{2n+1} - r_{2n})(r_{2n}+1) + r_{2n} - 1) \\
&= \frac{kr_{2n+1}^2 + (k+l+1)r_{2n+1} - (l-1)}{2l},
\end{aligned}
\end{equation}
where the second line follows from \eqref{eqn:2_1_recur_even} and the last line follows from \eqref{eqn:r_dioph_pair_odd}. Similarly, by \eqref{eqn:2_1_recur_odd} and \eqref{eqn:r_dioph_pair_even},
\begin{equation}\label{eqn:fn_odd_simple}
f_{2n-1} = \frac{l r_{2n}^2 + (k+l+1) r_{2n} - (k-1)}{2k}.
\end{equation}

By \eqref{eqn:keyformula},
\begin{align*}
\max_m \{m: & c_m(E(1,\frac{k}{l})) \leq r_{2n}\} = L_{\T_{k,l}}(lr) - 1 \\
&= \frac{l r_{2n}^2 + (k+l+1)r_{2n}}{2k} + \frac{1}{4}\left(1 + \frac{1}{k} + \frac{1}{l}\right) \\
&+ \frac{1}{12} \left(\frac{k}{l} + \frac{l}{k} + \frac{1}{kl}\right) + s_{-lr_{2n}}(l,1;k) + s_{-lr_{2n}}(k,1;l) - 1.
\end{align*}
For $n$ even, this is equal to $f_{2n-1}$ if
\begin{equation}\label{eqn:const_terms_eq_even}
\frac{k+1}{2k} = \frac{1}{4} \left(1 + \frac{1}{k} + \frac{1}{l}\right) + \frac{1}{12} \left(\frac{k}{l} + \frac{l}{k} + \frac{1}{kl} \right) + s_{-lr_{2n}}(l,1;k) + s_{-lr_{2n}}(k,1;l).
\end{equation}
Similarly, \eqref{eqn:beta_idx} holds for $n$ odd if
\begin{equation}\label{eqn:const_terms_eq_odd}
\frac{l+1}{2l} = \frac{1}{4} \left(1 + \frac{1}{k} + \frac{1}{l}\right) + \frac{1}{12} \left(\frac{k}{l} + \frac{l}{k} + \frac{1}{kl} \right) + s_{-kr_{2n+1}}(l,1;k) + s_{-kr_{2n+1}}(k,1;l).
\end{equation}
Induction on \eqref{eqn:2_1_recur_odd} and \eqref{eqn:2_1_recur_even} gives $2 \not| r(2,1)_{2n}$, $3 \not| r(3,2)_{2n}$ and $2 \not| r(3,2)_{2n+1}$. By direct computation, \eqref{eqn:const_terms_eq_even} and \eqref{eqn:const_terms_eq_odd} hold for each $(k,l) \in \{(1,1), (2,1), (3,2)\}$. This completes the proof of Theorem~\ref{thm:setheorem}.

\section{Proof of Theorem~\ref{thm:periodcollapse2}}

We conclude by proving Theorem~\ref{thm:periodcollapse2}.  Assume throughout that $(k,l)$ satisfies \eqref{eqn:kl_triplet}, and continue the notation of the previous section by letting $r_n$ denote $r(k,l)_n$.

We first prove the ``if" statements of Theorem~\ref{thm:periodcollapse2}.  If $(p,q)=(r_{2n\pm1},r_{2n})$ then, as explained in \S\ref{sec:classification}, $kp^2$ and $lq^2$ are relatively prime and $(k,l,p,q)$ satisfies \eqref{eqn:diophantine}.  Thus, by Theorem~\ref{thm:periodcollapse} $\mathcal{T}_{\frac{q}{kp},\frac{p}{ql}}$ is Ehrhart equivalent to $\mathcal{T}_{\frac{1}{k},\frac{1}{l}}$ and so $\mathcal{T}_{\frac{q}{kp},\frac{p}{ql}}$ has period $kl$.  Similarly, if $(p,q)=(lr(k,l)_{2n},kr(k,l)_{2n\pm1})$ then for 
\[(p',q') \eqdef (\frac{p}{l},\frac{q}{k})\]
$kq'^2$ and $lp'^2$ are relatively prime and $(k,l,q',p')$ satisfies \eqref{eqn:diophantine}.  Hence, by Theorem~\ref{thm:periodcollapse}, $\mathcal{T}_{\frac{p'}{kq'},\frac{q'}{lp'}}$ is Ehrhart equivalent to $\mathcal{T}_{\frac{1}{k},\frac{1}{l}}$.  Thus, $\mathcal{T}_{\frac{q}{kp},\frac{p}{ql}}$ has period $kl$, since $\mathcal{T}_{\frac{q}{kp},\frac{p}{ql}}$ is Ehrhart equivalent to $\mathcal{T}_{\frac{p'}{kq'},\frac{q'}{lp'}}$.

Assume now in addition that $(k,l) \in \lbrace (1,1), (2,1) \rbrace$. We now complete the proof of Theorem~\ref{thm:periodcollapse} by proving the ``only if" statements.  If $kp^2$ and $lq^2$ are relatively prime and $\mathcal{T}_{\frac{q}{kp},\frac{p}{lq}}$ has period $kl$, then by \eqref{eqn:keyformula} we must have
\begin{equation}
\label{eqn:keyperiodcollapseequation}
s_{klpq}(kp^2,1;lq^2)+s_{klpq}(lq^2,1;kp^2)=s_0(kp^2,1;lq^2)+s_0(lq^2,1;kp^2).
\end{equation}
We know in addition that $klpq \le kp^2 + lq^2$.  Hence, we can apply Lemma~\ref{lem:reciprocity} and Lemma~\ref{lem:reciprocity2} to \eqref{eqn:keyperiodcollapseequation} to conclude that $(k,l,p,q)$ satisfies \eqref{eqn:diophantine}, so the ``only if" direction of Theorem~\ref{thm:periodcollapse2} follows by Proposition~\ref{prop:classification}.

If $kp^2$ and $lq^2$ are not relatively prime, then we must have $(k,l)=(2,1)$ and it must also be the case that $q$ is divisible by $2$ and $p$ is not divisible by $2$.  Define $q'\eqdef\frac{q}{2}$.  We know that $2q'^2$ and $p^2$ are relatively prime.  Moreover, $\mathcal{T}_{\frac{q}{2p},\frac{p}{q}}$ is Ehrhart equivalent to $\mathcal{T}_{\frac{p}{2q'},\frac{q'}{p}}$.  If $\mathcal{T}_{\frac{p}{2q'},\frac{q'}{p}}$ has period $2$ then by \eqref{eqn:keyformula} we must have
\begin{equation}
\label{eqn:otherkeyperiodcollapseequation}
s_{2pq'}(2p^2,1;q'^2)+s_{2pq'}(q'^2,1;2p^2)=s_0(2p^2,1;q'^2)+s_0(q'^2,1;2p^2).
\end{equation}
Since $2pq' \le 2p^2 + q'^2,$ we can apply Lemma~\ref{lem:reciprocity} and Lemma~\ref{lem:reciprocity2} to \eqref{eqn:otherkeyperiodcollapseequation} to conclude that $(2,1,q',p)$ satisfies \eqref{eqn:diophantine}.  Theorem~\ref{thm:periodcollapse2} again then follows by Proposition~\ref{prop:classification}.
  
\begin{appendix}

\section{Appendix}

The purpose of this appendix is to explain how Theorem~\ref{thm:mcduffhofer} implies the following:  

\begin{theorem}
\label{thm:appendixtheorem}
Let $(k,l) \in \lbrace (1,1), (2,1), (3,2) \rbrace.$
Then 
\begin{equation}
\label{eqn:volumerider}
c(a,\frac{k}{l}) = \sqrt{\frac{al}{k}}
\end{equation}
for all $a \ge \frac{k}{l} (1+ \frac{l+1}{k})^2.$
\end{theorem}

Recall from the introduction that the right hand side of equation \eqref{eqn:volumerider} represents the volume obstruction.  We include Theorem~\ref{thm:appendixtheorem} to complement Theorem~\ref{thm:setheorem}. 

\begin{remark}
\label{rm:whythosekl}
 The method in the proof of Theorem~\ref{thm:appendixtheorem} can be adapted to establish equations like \eqref{eqn:volumerider} for other $k$ and $l$.  As $k$ and $l$ vary, all one needs in the proof is a bound like \eqref{eqn:ehrhartbound}, which can often be found by direct computation using \eqref{eqn:keyformula}.  
 \end{remark}

To place Theorem~\ref{thm:appendixtheorem} in its appropriate context, note that McDuff and Schlenk prove $c(a,1) = \sqrt{a}$ for all $a \geq \left(\frac{17}{6}\right)^2$, and Frenkel and M\"uller prove $c(a,2) = \sqrt{a/2}$ for all $a \geq \frac{1}{2}\left(\frac{15}{4}\right)^2$. Both proofs use methods that differ from ours.
%\eqref{eqn:volumerider} for $c(a,1)$ for $a \ge (\frac{17}{6})^2$, and Frenkel and M\"uller prove \eqref{eqn:volumerider} for $c(a,2)$ for $a \ge \frac{1}{2}(\frac{15}{4})^2$.  Both McDuff-Schlenk and Frenkel-M\"uller use methods that differ from ours.
In \cite[Thm. 1.3]{hind}, Buse and Hind show that for any $k,l$ with $k \geq l$, \eqref{eqn:volumerider} holds for all %$\op{Int}E(1,a)$ admits a volume filling embedding into a scaling of $E(1,\frac{k}{l})$ if, for $\beta = \frac{k}{l},$ 
\begin{equation}\label{eqn:hindbuse}
a \geq \frac{k}{l}\left(\frac{5}{4} + \frac{4l}{k}\right)^2. 
\end{equation}
When Theorem~\ref{thm:appendixtheorem} applies, it gives sharper bounds than \eqref{eqn:hindbuse}. In particular, for $(k,l)=(3,2)$, \eqref{eqn:hindbuse} gives the bound $a \geq \frac{2209}{96} \approx 23.01$, while Theorem~\ref{thm:appendixtheorem} gives the bound $a \geq 6$.
\begin{proof}
If $a_k$ and $b_k$ are two sequences (indexed with $k$ starting at $0$), define a new sequence  
\[(a \# b)_k \eqdef \op{sup}_{k_1+k_2 = k} a_{k_1}+b_{k_2}. \]
This is the {\em sequence sum} operation originally defined by Hutchings in \cite{qech}.

Let $\mathcal{N}(a,b)$ be the sequence whose $k^{th}$ term is $c_k(E(a,b))$.  For the sequences $\mathcal{N}(a,b)$, the sequence sum operation satisfies the identity
\[ \mathcal{N}(a,b)\#\mathcal{N}(a,c) = \mathcal{N}(a,b+c),\]
for $a, b,$ and $c$ any positive integers.  This is proven by an elementary argument in \cite[Lem. 2.4]{hofer}.  It follows, see \cite[\S 2]{hofer}, that if $a$ is rational then there is a finite sequence of positive rational numbers $(a_1,\ldots,a_n)$ associated to $a$, called a {\em weight sequence} for $a$, such that 
\begin{equation}
\label{eqn:aismall} 
 a_i \le 1 \quad \forall i,
\end{equation}
\begin{equation}
\label{eqn:weightsequenceequation}
 \sum_{i=1}^{n} a_i^2 = a,
\end{equation} 
and
\begin{equation}
\label{eqn:ballpacking}
\mathcal{N}(1,a)=\mathcal{N}(a_1,a_1)\#\ldots\#\mathcal{N}(a_n,a_n).
\end{equation}
The weight sequence is closely related to the continued fraction expansion for $a$, see \cite[\S 2]{hofer}.
  
By \eqref{eqn:ballpacking}, to prove Theorem~\ref{thm:appendixtheorem}, it suffices to show that 
\begin{equation}
\label{eqn:needtoshow}
\sum_{i=1}^{n} ld_ia_i \le d\sqrt{\frac{al}{k}}
\end{equation}
whenever $(d_1,\ldots,d_n,d)$ are nonnegative integers %, not all zero, 
satisfying 
\begin{equation}
\label{eqn:whatitsatisfies}
\sum_{i=1}^{n} \frac{d_i^2 + d_i}{2}+1 \le L_{\mathcal{T}_{1/k,1/l}}(d).
\end{equation}
For $(k,l) \in \lbrace (1,1), (2,1), (3,2) \rbrace$, we know by direct computation that 
\begin{equation}
\label{eqn:ehrhartbound}
L_{\mathcal{T}_{1/k,1/l}}(d) \le \frac{d^2}{2kl}+(\frac{1}{2k}+\frac{1}{2l}+\frac{1}{2kl})d+1.
\end{equation}

If 
\[\sum_{i=1}^{n} d_i^2 \le\frac{d^2}{kl},\] 
then \eqref{eqn:needtoshow} follows by applying \eqref{eqn:weightsequenceequation} and the Cauchy-Schwarz inequality.  Thus, if $(d_1,\ldots,d_n,d)$ satisfies \eqref{eqn:whatitsatisfies}, we can assume that 
\begin{equation}
\label{eqn:canassume}
\sum_{i=1}^{n} d_i \le  (\frac{1}{k}+\frac{1}{l}+\frac{1}{kl})d.
\end{equation}
Then by \eqref{eqn:aismall},% \eqref{eqn:canassume} gives
\begin{equation}
\label{eqn:almostthere}
\sum_{i=1}^n l d_i a_i \le (1+\frac{l+1}{k}) d.
\end{equation}
Since $(1+\frac{l+1}{k}) \le \sqrt{\frac{al}{k}}$ if $a \ge \frac{k}{l} (1+ \frac{l+1}{k})^2,$ Theorem~\ref{thm:appendixtheorem} follows.
\end{proof}
 
\end{appendix}

\end{document}